\documentclass[10pt,a4paper]{amsart}

\theoremstyle{plain}
\newtheorem{theorem}{Theorem}[section]
\newtheorem{lemma}[theorem]{Lemma}

\theoremstyle{definition}

\theoremstyle{remark}
\newtheorem{remark}[theorem]{Remark}

\newcommand{\hc}{\ensuremath{\mathcal{H}}}

\def\bin #1#2 {\left( \matrix { #1 \cr #2 \cr } \right) }

\begin{document}

\title[On the topology of a resolution of  isolated singularities, II]
{On the topology of a resolution of isolated singularities, II}

\author{Vincenzo Di Gennaro }
\address{Universit\`a di Roma \lq\lq Tor Vergata\rq\rq, Dipartimento di Matematica,
Via della Ricerca Scientifica, 00133 Roma, Italy.}
\email{digennar@axp.mat.uniroma2.it}

\author{Davide Franco }
\address{Universit\`a di Napoli
\lq\lq Federico II\rq\rq, Dipartimento di Matematica e
Applicazioni \lq\lq R. Caccioppoli\rq\rq, Via Cintia, 80126
Napoli, Italy.} \email{davide.franco@unina.it}

\abstract  Let $Y$ be a complex projective  variety of dimension $n$
with isolated singularities, $\pi:X\to Y$ a resolution of
singularities, $G:=\pi^{-1}\left(\rm{Sing}(Y)\right)$ the
exceptional locus. From the Decomposition Theorem one knows that the
map $H^{k-1}(G)\to H^k(Y,Y\backslash {\rm{Sing}}(Y))$ vanishes for
$k>n$. It is also known that, conversely, assuming this vanishing
one can prove the Decomposition Theorem for $\pi$ in few pages. The
purpose of the present paper is to exhibit a direct proof of the
vanishing. As a consequence, it follows a complete and short proof
of the  Decomposition Theorem for $\pi$, involving only ordinary
cohomology.



\endabstract
\maketitle

\section{Introduction}
Consider an $n$-dimensional  integral complex projective variety $Y$
with {\it isolated singularities}. Fix a {\it resolution of
singularities} $\pi: X\rightarrow Y$ of $Y$. This means that $X$ is
an irreducible and smooth  projective variety, and $\pi$ is a
birational morphism inducing an isomorphism $X\backslash
\pi^{-1}({\rm{Sing}}(Y))\cong Y\backslash {\rm{Sing}}(Y)$. In this case, the celebrated
{\it Decomposition Theorem} of Deligne, Gabber, Beilinson, and Bernstein \cite{BBD}, assumes the following form:

\begin{theorem}[The Decomposition Theorem for varieties with isolated singularities]
\label{DecThm} In $D^b(Y)$, we have a decomposition
$$R\,\pi_*\mathbb Q_X\cong IC^{\bullet}_Y[-n]\oplus \mathcal H^{\bullet},$$
where $\mathcal H^{\bullet}$ is quasi isomorphic to a skyscraper
complex on ${\rm{Sing}}(Y)$. Furthermore,  we have
\begin{enumerate}
\item $\hc^k(\mathcal H^{\bullet})\cong H^k(G)$, for all $k\geq n$,
\item $\hc^k(\mathcal H^{\bullet})\cong H_{2n-k} (G)$, for all $k< n$,
\end{enumerate}
where  $G:=\pi ^{-1}({\rm{Sing}}(Y))$, and $H^k(G)$ and
$H_{2n-k}(G)$ have $\mathbb Q$-coefficients.
\end{theorem}

In \cite{GMtm} Goresky and MacPherson remarked that from previous
theorem one deduces the following vanishing, concerning ordinary
cohomology (compare also with \cite[(1.11) Theorem]{Steenbrink}):
\begin{equation}
\label{vanishing} H^{k-1}(G)\to H^k(Y,Y\backslash {\rm{Sing}}(Y)) \hskip2mm \textrm{vanishes
for} \hskip2mm k>n.
\end{equation}
More recently, in \cite{DGF3} we observed that, conversely, assuming
the vanishing (\ref{vanishing}) one can prove Theorem \ref{DecThm}
in few pages \cite[Theorem 3.1]{DGF3}.

\smallskip
Continuing \cite{DGF3}, in the present paper we give a direct proof of the vanishing (\ref{vanishing}), without using
Theorem \ref{DecThm}. Therefore, combining with \cite[Theorem 3.1]{DGF3}, it follows
a complete and short proof of the  Decomposition Theorem for $\pi$,
involving only ordinary cohomology.

\smallskip
Our proof of the vanishing (\ref{vanishing}) relies on an argument
similar to that developed by Navarro in  \cite[(5.1)
Proposition]{N}. First, using certain preliminary results we already
stated in \cite[Lemma 4.1 and Lemma 4.2]{DGF3}, we reduce to the
case the exceptional locus $G=\pi ^{-1}({\rm{Sing}}(Y))$ is a {\it
simple normal crossing divisor} \cite[p. 240]{L1}, and $X$ admits an
ample line bundle of the form $\pi^*(\mathcal L)\otimes \mathcal
O_X(D)$, where $\mathcal L$ is an ample line bundle on $Y$, and $D$
a divisor supported on $G$ (see Lemma \ref{resolutionf} and Lemma
\ref{reduction} below). Next, we conclude using again \cite[Lemma
4.1]{DGF3}, general properties of mixed Hodge Theory, and a slight
generalization of a Lemma of Steenbrink appearing in \cite[p.
288]{N} (see Lemma \ref{b}, Lemma \ref{surnn} and Lemma
\ref{Steenbrinknn} below).

\section{The proof of the vanishing}

We need the  following two lemmas. The first one is certainly
well known, but we prove it for lack of a suitable reference.

\begin{lemma}\label{resolutionf}
Let $\pi:X\to Y$ be a resolution of singularities of $Y$. Then
there exists a resolution of singularities $p: Z\to Y$ of $Y$
satisfying the following conditions:

\smallskip
(1) there is a morphism $\pi_X:Z\to X$ such that $p=\pi\circ
\pi_X$;

\smallskip
(2) $\Gamma :=p ^{-1}({\rm{Sing}}(Y))$ is a simple normal crossing
divisor (s.n.c.) on $Z$;

\smallskip
(3) for every ample line bundle $\mathcal L$ on $Y$, there are
integers $a,a_1,\dots,a_r$ ($a>0$) such that $p^*(\mathcal
L^{\otimes a}) \otimes \mathcal
O_Z\left(\sum_{i=1}^{r}a_i\Gamma_i\right)$ is an ample line bundle
on $Z$, the sum being taken over the components of $\Gamma$.
\end{lemma}

\begin{proof} Let $\pi_1:X_1\to Y$ be a resolution of singularities
verifying condition (2), i.e. such that $\pi_{1}
^{-1}({\rm{Sing}}(Y))$ is a s.n.c. divisor \cite[Theorem 4.1.3]{L1}.
One can construct $\pi_1:X_1\to Y$ via a sequence $X_1=B_{h}\to
B_{h-1}\to\dots \to B_1=Y$ of blowings-up along smooth centers
supported in the singular locus of $Y$ \cite[loc. cit.]{L1}. Fix an
ample line bundle $\mathcal L$ on $Y$, and an integer $h>0$ such
that $\mathcal L^{\otimes h}$ is very ample, corresponding to a
closed immersion $i:Y\subset \mathbb P$ of $Y$ in some projective
space $\mathbb P$ ($\mathcal L^{\otimes h}= i^*(\mathcal O_{\mathbb
P}(1))$). Each element $B_j$ of this sequence is contained in an
element $C_j$ of a sequence $C_{h}\to C_{h-1}\to\dots \to
C_1=\mathbb P$ of blowings-up along the {\it same} smooth centers,
starting from $\mathbb P$. By \cite[Proposition 6.7, (e)]{FultonIT},
we know that the Picard group of each $C_j$ is generated by the
pull-back of the hyperplane class of $\mathbb P$, and certain
divisor classes supported in the singular locus of $Y$. Therefore,
an ample bundle $\mathcal M$ on $C_h$ is {\it necessarily} the
pull-back of a positive power of $\mathcal O_{\mathbb P}(1)$,
tensored with a line bundle like $\mathcal O_{C_h}(E)$, with $E$
divisor supported in the singular locus of $Y$. Restricting such
ample bundle $\mathcal M$ on $X_1$, we get a line bundle as in (3).
This proves that there exists a  resolution $\pi_1:X_1\to Y$
satisfying conditions (2) and (3).

Now consider the fibred product $X\times_{Y} X_1$. It contains
$U:=Y\backslash {\rm{Sing}}(Y)$. Let  $\overline{U}$ be the closure
of $U$ in $X\times_{Y} X_1$. Applying \cite[loc. cit.]{L1} to the
blowing-up of $\overline{U}$ along $\overline{U}\backslash U$, we
may construct a resolution of singularities $\varphi: Z\to
\overline{U}$ of $\overline{U}$, inducing an isomorphism
$\varphi^{-1}(U)\cong U$, and such that
$\Gamma:=\varphi^{-1}(\overline{U}\backslash U)$ is a s.n.c. divisor
on $Z$. Composing  $\varphi: Z\to \overline{U}$ with the inclusion
$\overline{U}\subset X\times_{Y} X_1$,  and the projections
$X\times_{Y} X_1\to X$ and $X\times_{Y} X_1\to X_1$, we get maps
$\pi_{X_1}:Z\to X_1$, $\pi_X:Z\to X$ and $p:Z\to Y$, with
$p=\pi_1\circ \pi_{X_1} =\pi\circ \pi_X$:
$$
\begin{array}{ccccc}
& & Z \quad & & \\
\quad\quad\quad\quad\stackrel{\pi_{X_1}}\swarrow  & & &\searrow\stackrel{\pi_X}{}\\
 X_1  & & \downarrow\stackrel{p}{}\quad  & & X\\
\quad\quad\quad\quad\stackrel{\pi_1}{}\searrow&\,& &\swarrow\stackrel{\pi}{}\\
&& Y. \quad&\\
\end{array}
$$
The morphism  $p:Z\to Y$ is the map we are looking for. In fact,
$p:Z\to Y$ is a resolution of singularities of $Y$. Moreover, it
satisfies conditions (1) and (2) because $p=\pi\circ \pi_X$ and
$\Gamma=\varphi^{-1}(\overline{U}\backslash
U)=p^{-1}({\rm{Sing}}(Y))$ is a s.n.c. divisor. It remains to check
condition (3) for $p$.

To this aim, fix an ample bundle $\mathcal L$ on $Y$. Since
$\pi_1:X_1\to Y$ satisfies  condition $(3)$, for some $a>0$, the
line bundle $\pi_1^*(\mathcal L^{\otimes a})$, tensored with a
suitable line bundle supported in the singular locus of $Y$, is an
ample bundle $\mathcal L_1$ on $X_1$. Since $\pi_{X_1}:Z\to X_1$ is
a birational morphism between {\it smooth} projective varieties, by
\cite[Corollary 4.1.4 and Remark 4.1.5]{L1} we know that, for some
$b>0$, the line bundle $\pi_{X_1}^*(\mathcal L_1^{\otimes b})$,
tensored with a suitable line bundle supported in the singular locus
of $Y$, is an ample line bundle on $Z$. Since $p=\pi_1\circ
\pi_{X_1}$, it follows that also $p:Z\to Y$ satisfies condition
$(3)$.
\end{proof}

\begin{lemma}\label{reduction}
Let $\pi:X\to Y$ and $p:Z\to Y$ be resolutions of singularities of
$Y$. Set $U:=Y\backslash {\rm{Sing}}(Y)$, $G :=\pi
^{-1}({\rm{Sing}}(Y))$  and $\Gamma :=p ^{-1}({\rm{Sing}}(Y))$.
Assume there is a  morphism $\pi_X:Z\to X$ such that $p=\pi\circ
\pi_X$. Fix an integer $k>n$. If the map $H^{k-1}(\Gamma)\to
H^k(Y,U)$ vanishes, then also the map $ H^{k-1}(G)\to H^k(Y,U)$
vanishes.
\end{lemma}

\begin{proof}
Consider the following commutative diagram given by
the pull-back:
$$
\begin{array}{ccccc}
H^{k-1}(Y) &\stackrel{\pi^*}{\longrightarrow}  & H^{k-1}(X)&\stackrel{\pi_X^*}{\longrightarrow}  & H^{k-1}(Z)\\
\quad\quad\stackrel{\alpha^*_{k-1}}{}\searrow&\,&\downarrow\stackrel{\beta^*_{k-1}}{}&\swarrow\stackrel{\gamma^*_{k-1}}{}\\
&&H^{k-1}(U).&\\
\end{array}
$$
If the map $H^{k-1}(\Gamma)\to H^k(Y,U)$ vanishes,
from \cite[Lemma 4.1 and 4.2]{DGF3} it follows that
$\Im(\alpha^*_{k-1})=\Im(\gamma^*_{k-1})$. Since
previous diagram commutes, we also have
$\Im(\alpha^*_{k-1})\subseteq\Im(\beta^*_{k-1})\subseteq \Im (\gamma^*_{k-1})$.
Therefore we get $\Im(\alpha^*_{k-1})=\Im(\beta^*_{k-1})$.
Again from \cite[loc. cit.]{DGF3}, it follows that the map
$H^{k-1}(G)\to H^k(Y,U)$ vanishes.
\end{proof}

By previous Lemma \ref{resolutionf} and Lemma \ref{reduction}, in
order to prove the vanishing (\ref{vanishing})  we may assume that
$G$ is a s.n.c. divisor, and that there exists a very ample line
bundle $\mathcal L$ on $Y$, and integers $a_1,\dots,a_r$, such that
\begin{equation}\label{defM}
\mathcal M:=\pi^*(\mathcal L) \otimes \mathcal
O_X\left(\sum_{j=1}^{r}a_jG_j\right)
\end{equation}
is a very ample line bundle on $X$, the sum being taken over the
components of $G$. Denote by $[\mathcal M]\in H^2(X)$ its cohomology
class, and by $\mu\in H^2(G)$ its restriction to $G$. Denote by
$$
\eta_i:H_{n-i}(G)\to H^{n+i}(G),
$$
the  map obtained composing  the pull-back $H^{n+i}(X)\to
H^{n+i}(G)$, with the isomorphism induced by Poincar\'e duality
$H_{n-i}(X)\to H^{n+i}(X)$, and the push-forward $H_{n-i}(G)\to
H_{n-i}(X)$. By \cite[(3) p. 198]{DGF3} and \cite[Lemma 14, p.
351]{Spanier} we see that, for a fixed $i\geq 0$, to prove the
vanishing of the map $H^{n+i}(G)\to H^{n+i+1}(Y,U)$ ($U:=Y\backslash
{\rm{Sing}}(Y)$)  is equivalent to prove that the map $\eta_i$ is
onto. Now consider the map
$$
\epsilon_i:H_{n+i}(G)\to H^{n+i}(G)
$$
obtained composing $\eta_i:H_{n-i}(G)\to H^{n+i}(G)$ with the
cap-product $H_{n+i}(G)  \stackrel{\mu^i \cap
\,\cdot}{\rightarrow} H_{n-i}(G)$:
$$
\begin{array}{ccccc}
H_{n+i}(G) & & \stackrel{\epsilon_i}{\longrightarrow}  & & H^{n+i}(G)\\
\quad\quad\quad\quad\quad\stackrel{\mu^i \cap \,\cdot}{}\searrow&\,& &&\nearrow\,\stackrel{}{\eta_i}\\
 &&H_{n-i}(G).&&
\end{array}
$$
In order to prove that the map $\eta_i:H_{n-i}(G)\to H^{n+i}(G)$ is
onto, it suffices to prove that the map $\epsilon_i$ is an
isomorphism. Summing up, {\it in order to prove the vanishing
(\ref{vanishing}), it suffices to prove that the map
$\epsilon_i:H_{n+i}(G)\to H^{n+i}(G)$ is an isomorphism for all
$i\geq 0$}.

\begin{remark}\label{remarkPHL2n} For all $y\in {\rm{Sing}}(Y)$, set $G_y:=\pi^{-1}(y)$.
By \cite[(1.4) Corollaire]{N} we see that, for $1\leq i< n$,
$${}^p\mathcal H^{-i}(R\,\pi_*\mathbb Q_X[n])_y\cong
H_{n+i}(G_y)\quad {\text{and}}\quad {}^p\mathcal
H^{i}(R\,\pi_*\mathbb Q_X[n])_y\cong H^{n+i}(G_y).$$ Therefore, to
prove that the map $\epsilon_i:H_{n+i}(G)\to H^{n+i}(G)$ is an
isomorphism for $i\geq 1$, is equivalent to prove perverse
Hard-Lefschetz Theorem for $\pi$ and $\mathcal M$ \cite[(6.1)
Proposition]{N}, \cite[Theorem 3.3.1., p. 573]{DeCM3},
\cite[Theorem 5.4.8, p. 160]{Dimca2}.
\end{remark}

\medskip
Now we are going to prove that the map $\epsilon_i$ is an
isomorphism for all $i\geq 0$.

\smallskip
To this purpose, let $\widetilde G\to G$ be a {resolution} of $G$
\cite[p. 109, and Definition 5.14, p.119]{PSt}. And consider the
following commutative diagram:
$$
\begin{array}{ccccc}
H_{n+i}(\widetilde G) & & \stackrel{a}{\longrightarrow}  & & H_{n+i}(G)\\
\quad\quad\quad\quad\quad\searrow\stackrel{\gamma}{}&\,& &&\stackrel{b}{}\,\swarrow\quad\quad\quad\\
\stackrel{\beta}{}\downarrow&&H_{n+i}(X)&&\quad\downarrow\stackrel{\epsilon_i}{}\\
\quad\quad\quad\quad\quad\swarrow\stackrel{\rho}{}&\,& &&\stackrel{c}{}\,\searrow\quad\quad\quad\\
H^{n+i}(\widetilde G) & & \stackrel{d}{\longleftarrow} & &
H^{n+i}(G),
\end{array}
$$
where:

\smallskip
$\bullet$ the maps $a$, $b$ and $\gamma$ are the push-forward, and
$d$ is the pull-back;

\smallskip
$\bullet$ the map $c$ is the composition of the pull-back
$H^{n+i}(X)\to H^{n+i}(G)$, with the cup-product $H^{n-i}(X)
\stackrel{[\mathcal M]^i\,\cup \cdot}{\longrightarrow}
H^{n+i}(X)$, and the isomorphism $PD^{-1}:H_{n+i}(X)\to
H^{n-i}(X)$ induced by Poincar\'e Duality;

\smallskip
$\bullet$ the map $\rho$ is the composition of the pull-back
$H^{n+i}(X)\to H^{n+i}(\widetilde G)$, with the cup-product
$H^{n-i}(X) \stackrel{[\mathcal M]^i\,\cup \cdot}{\longrightarrow}
H^{n+i}(X)$, and the isomorphism $PD^{-1}:H_{n+i}(X)\to
H^{n-i}(X)$ induced by Poincar\'e Duality;

\smallskip
$\bullet$ $\beta:=\rho \circ \gamma$.

\smallskip
We need the following lemmas. We prove them in a moment.

\begin{lemma}\label{b}
The push-forward map $b:H_{n+i}(G)\to H_{n+i}(X)$ is injective.
\end{lemma}

\begin{lemma}\label{surnn}
The push-forward map $a:H_{n+i}(\widetilde G)\to H_{n+i}(G)$ is
onto.
\end{lemma}

\begin{lemma}\label{Steenbrinknn}
$\ker \beta\subseteq \ker \gamma$.
\end{lemma}

\medskip\noindent
Taking into account previous lemmas, a simple diagram chase proves
that $\epsilon_i$ is injective, hence an isomorphism because
$H_{n+i}(G)$ and $H^{n+i}(G)$ have the same dimension. In fact,
assume $\epsilon_i(x)=0$. Let $y\in H_{n+i}(\widetilde G)$ such that
$a(y)=x$. Then $d(\epsilon_i(a(y)))=0$, i.e. $\beta (y)=0$. Then
$\gamma(y)=0$, and therefore $b(a(y))=0$. It follows that $a(y)=0$,
i.e. $x=0$.

\medskip
To conclude, we are going to prove previous lemmas. We only prove
Lemma \ref{surnn} and Lemma \ref{Steenbrinknn} because Lemma \ref{b}
follows from \cite[Lemma 4.1]{DGF3}. We may assume all cohomology
and homology groups are with $\mathbb C$-coefficients.

\medskip
\begin{proof}[Proof of Lemma \ref{surnn}]
Since $X$ is smooth, by \cite[Proposition 4.20, p. 102]{PSt} we
know that $H^{n+i}(X)$ has no weights of order $n+i-1$, i.e.
$W_{n+i-1}H^{n+i}(X)=0$. Since the pull-back $H^{n+i}(X)\to
H^{n+i}(G)$ is onto \cite[Lemma 4.1]{DGF3}, and is a morphism of
mixed Hodge structure, it follows that also $H^{n+i}(G)$ has no
weights of order $n+i-1$, i.e. $W_{n+i-1}H^{n+i}(G)=0$
\cite[Corollary 3.6, p. 65, and Theorem 5.33, (iii), p. 126]{PSt}.
On the other hand, since $\widetilde G\to G$ is a resolution (and
$\widetilde G$ and $G$ are projective), we have
$W_{n+i-1}H^{n+i}(G)=\ker(H^{n+i}(G)\to H^{n+i}(\widetilde G))$
\cite[Corollary 5.42, p. 133, and Remark 5.15, 1), p. 119]{PSt}.
Therefore, the pull-back $H^{n+i}(G)\to H^{n+i}(\widetilde G)$ is
injective. This is equivalent to say that the push-forward map
$H_{n+i}(\widetilde G)\to H_{n+i}(G)$ is onto.
\end{proof}

\medskip
\begin{proof}[Proof of Lemma \ref{Steenbrinknn}]
Since $G$ is a s.n.c. divisor, its irreducible components $G_1$,
$\dots$, $G_r$ are smooth, and $\widetilde G$ is simply the disjoint
union of them. Via Poincar\'e Duality  on each components of $G$,
and on $X$, we may identify $H_{n+i}(\widetilde G)$ with
$H^{n-i-2}(\widetilde G)$, $H_{n+i}(X)$ with $H^{n-i}(X)$, and the
push-forward $\gamma: H_{n+i}(\widetilde G)\to H_{n+i}(X)$ with a
Gysin map $\gamma':H^{n-i-2}(\widetilde G)\to H^{n-i}(X)$. Hence, to
prove that $\ker \beta\subseteq \ker \gamma$ is equivalent to prove
that
$$\ker \beta'\subseteq \ker \gamma',$$ where $\beta':=\rho'\circ
\gamma'$, and $\rho'$ denotes the composition of the pull-back
$H^{n+i}(X)\to H^{n+i}(\widetilde G)$, with the cup-product
$H^{n-i}(X) \stackrel{[\mathcal M]^i\,\cup \cdot}{\longrightarrow}
H^{n+i}(X)$:
$$
\begin{array}{ccccc}
& & H^{n-i-2}(\widetilde G) \quad & & \\
\quad\quad\quad\quad  & & &\searrow\stackrel{\gamma'}{}\\
 & & \downarrow\stackrel{\beta'}{}\quad  & & H^{n-i}(X)\\
\quad\quad\quad\quad&\,& &\swarrow\stackrel{\rho'}{}\\
&& H^{n+i}(\widetilde G). \quad&\\
\end{array}
$$

Let $v\in \ker\beta'$.

\smallskip\noindent
Since the Gysin map and the pull-back are morphisms of Hodge
structures \cite[Corollary 1.13, p. 17, and  Lemma 1.19, p.19]{PSt},
we may assume that $v\in H^{p,q}(\widetilde G)$, with $p+q=n-i-2$.
Set:
$$
\lambda:H^{n-i}(X)\to H^{n+i+2}(X), \quad \lambda(x):=x\cup
[\mathcal M]^{i+1},
$$
$$
\epsilon:H^{n-i}(X)\to H^{n+i+2}(X), \quad \epsilon(x):=x\cup
\left([\mathcal M]^{i+1}-\left[\mathcal
O_X\left(\sum_{j=1}^{r}a_jG_j\right)\right]^{i+1}\right),
$$
$$
\delta:H^{n-i}(X)\to H^{n+i+2}(X), \quad \delta(x):=x\cup
\left[\mathcal O_X\left(\sum_{j=1}^{r}a_jG_j\right)\right]^{i+1},
$$
$$
\gamma_j: H^{n+i}(G_j)\to H^{n+i+2}(X), \quad \gamma_j:={\text{the
Gysin map}}.
$$
We have:
\begin{equation}\label{comp}
\lambda(\gamma'(v))=(\epsilon+\delta)(\gamma'(v))=\epsilon(\gamma'(v))+\delta(\gamma'(v)).
\end{equation} Notice that, since the singular locus of $Y$ is finite,  we have:
\begin{equation}\label{LG1}
[\pi^*(\mathcal L)]\cup [\mathcal O_X(G_j)]=0\in H^4(X)
\end{equation}
for every component $G_j$ of $G$. Hence (compare with (\ref{defM})):
\begin{equation}\label{LG2}
[\mathcal M]^{i+1}= [\pi^*(\mathcal L)]^{i+1} +\left[\mathcal
O_X\left(\sum_{j=1}^{r}a_jG_j\right)\right]^{i+1}.
\end{equation}
Taking into account (\ref{LG1}), (\ref{LG2}), and the projection
formula \cite[p. 424]{PSt}, it follows that:
$$
\epsilon(\gamma'(v))=\gamma'(v)\cup [\pi^*(\mathcal L)]^{i+1}=0.
$$
Continuing previous computation (\ref{comp}), and using again
(\ref{LG1}) and (\ref{LG2}) as before, we get:
$$
\lambda(\gamma'(v))=\delta(\gamma'(v))=\left[\left(\sum_{j=1}^{r}a_j\gamma_j\right)\circ
\rho' \circ \gamma' \right](v)=
\left[\left(\sum_{j=1}^{r}a_j\gamma_j\right)\circ \beta'
\right](v)=0.
$$
This proves that, if $v\in\ker\beta'$, then $\gamma'(v)$ is a
primitive cohomology class \cite[p. 25 and 26]{PSt}.

\smallskip
Now denote by $f_j:G_j\to X$ the inclusion, that we may see as the
composition of the natural map $\widetilde G\to X$, with the
inclusion $l_j:G_j\to \widetilde G$. Denote by $v_1,\dots,v_r$ the
components of $v\in H^{n-i-2}(\widetilde G)=\oplus_{j=1}^{r}
H^{n-i-2}(G_j)$. We have:
$$
{\overline{\gamma'(v)}}\cup \gamma'(v)\cup [\mathcal M]^i=
{{\gamma'(\overline v)}}\cup \gamma'(v)\cup [\mathcal M]^i=
\sum_{j=1}^{r} {{\gamma'(\overline {v_j})}}\cup \gamma'(v)\cup
[\mathcal M]^i.
$$
By the projection formula  we may write:
$$
{{\gamma'(\overline {v_j})}}\cup \left(\gamma'(v)\cup [\mathcal
M]^i\right)={\overline {v_j}}\cup \left(f^*_j(\gamma'(v)\cup
[\mathcal M]^i)\right).
$$
Since
$$
f^*_j(\gamma'(v)\cup [\mathcal M]^i)=(l^*_j\circ \rho'\circ
\gamma')(v)=(l^*_j\circ \beta')(v)=0\in H^{n+i}(G_j),
$$
we deduce
$$
{\overline {v_j}}\cup \left(f^*_j(\gamma'(v)\cup [\mathcal
M]^i)\right)=0\in H^{2n-2}(G_j)\cong \mathbb C,
$$
and so
$$
{\overline{\gamma'(v)}}\cup \gamma'(v)\cup [\mathcal M]^i=0\in
H^{2n}(X)\cong \mathbb C.
$$
Summing up, $\gamma'(v)$ lies in $H^{p+1,q+1}(X)$, is primitive, and
${\overline{\gamma'(v)}}\cup \gamma'(v)\cup [\mathcal M]^i=0$. By
the Hodge-Riemann bilinear relations it follows that $\gamma'(v)=0$.
\end{proof}

\end{document}